\documentclass[10pt]{amsart}

\usepackage{amsmath,amssymb,amsthm,amsrefs,a4wide}
\usepackage{latexsym}
\usepackage{indentfirst}
\usepackage{graphicx}
\usepackage{placeins}
\usepackage{booktabs}
\usepackage{algorithm}
\usepackage{algorithmic}
\usepackage{multirow}

\theoremstyle{plain}
\newtheorem{theorem}{Theorem}

\newtheorem{claim}[theorem]{Claim}

\theoremstyle{definition}

\theoremstyle{remark}

\renewcommand{\d}{\mathrm{d}}

\newcommand{\diag}{\mathrm{diag}}

\newcommand{\eps}{\epsilon}

\newcommand{\norm}[1]{\lVert#1\rVert}

\newcommand{\bbm}{\begin{bmatrix}}
\newcommand{\ebm}{\end{bmatrix}}

\newcommand{\R}{\mathrm{R}}

\newcommand{\T}{\mathsf{T}}

\begin{document}

\title[Natural Gradient for Combined Loss]{Natural Gradient for Combined Loss Using Wavelets}

\author[]{Lexing Ying}
\address[Lexing Ying]{Department of Mathematics and ICME, Stanford University, Stanford, CA 94305}
\email{lexing@stanford.edu}

\thanks{The work of L.Y. is partially supported by the U.S. Department of Energy, Office of Science,
  Office of Advanced Scientific Computing Research, Scientific Discovery through Advanced Computing
  (SciDAC) program and also by the National Science Foundation under award DMS-1818449.
}


%\subjclass[2010]{65Z05, 82B28, 82B80.}

\begin{abstract}
  Natural gradients have been widely used in optimization of loss functionals over probability
  space, with important examples such as Fisher-Rao gradient descent for Kullback-Leibler
  divergence, Wasserstein gradient descent for transport-related functionals, and Mahalanobis
  gradient descent for quadratic loss functionals. This note considers the situation in which the
  loss is a convex linear combination of these examples. We propose a new natural gradient algorithm
  by utilizing compactly supported wavelets to diagonalize approximately the Hessian of the combined
  loss. Numerical results are included to demonstrate the efficiency of the proposed algorithm.
\end{abstract}

\maketitle

%-----------------------------------
\section{Introduction}\label{sec:intro}

Many problems in partial differential equations and machine learning can be formulated as
optimization problems over probability densities. For a domain $\Omega$, let $E(p)$ be a loss or
energy functional defined for the probability densities $p$ over $\Omega$. The goal is to find $p^*$
that minimizes $E(p)$. A common approach, especially for $E(p)$ with a unique minimum, is to follow
the gradient descent (GD) dynamics. However, depending the metric used in the gradient calculation,
different gradient descent algorithms exhibit drastically different convergence behavior. The term
{\em natural gradient} refers to the practice of choosing an appropriate metric depending on the
loss functional $E(p)$ as well as the probability space. Below are several well-known examples of
natural gradient.
\begin{itemize}
\item Wasserstein GD that scales the Euclidean gradient $\frac{\delta E}{\delta p}(p)$ with the
  metric $-\nabla\cdot(p\nabla)$. Wasserstein GD is typically effective for a loss $E(p)$ that
  behaves like the square of the 2nd Wasserstein distance.
\item Fisher-Rao GD that scales the Euclidean gradient  $\frac{\delta E}{\delta p}(p)$ with the diagonal
  tensor $\diag(p)$. Fisher-Rao GD is quite effective for a loss $E(p)$ such as the Kullback-Leibler
  divergence $\int p(x) \ln\frac{p(x)}{\mu(x)} \d x$.
\item Mahalanobis GD that scales the Euclidean gradient $\frac{\delta E}{\delta p}(p)$ with a positive
  definite metric $B$. Mahalanobis GD is efficient for a quadratic loss of the form
  $\frac{1}{2}(p-\mu, A(p-\mu))$ with $B \approx A^{-1}$. In this note, we consider the case that $A$
  is a positive semidefinite pseudo-differential operator, for example $A=-\Delta$.
\end{itemize}
A general principle from these examples is that, for a natural gradient to be effective, the metric
used at the density $p$ should be an approximate inverse of the Hessian of the loss $E(p)$ at
$p$. In each of these three examples, an approximate inverse of the Hessian can be derived quite
explicitly.

\subsection{Problem statement.}

In several problems from kinetic theory and statistical machine learning, one is faced with a loss
or energy functional $E(p)$ that is a linear combination of these three forms mentioned above, i.e.,
\[
E(p) = \alpha_1 E_1(p) + \alpha_2 E_2(p) + \alpha_3 E_3(p),
\]
where $\alpha_1, \alpha_2, \alpha_3 \ge 0$ and $E_1$, $E_2$, and $E_3$ are of the Wasserstein,
Fisher-Rao, and Mahalanobis types, respectively, i.e.,
\[
\frac{\delta^2 E_1}{\delta p^2}(p) \approx (-\nabla\cdot(p\nabla))^+,\quad
\frac{\delta^2 E_2}{\delta p^2}(p) \approx \diag\left(\frac{1}{p}\right), \quad
\frac{\delta^2 E_3}{\delta p^2}(p) \approx A
\]
where $(\cdot)^+$ stands for pseudo-inverse. As a result, the Hessian of $E(p)$ has the following
approximation
\[
\frac{\delta^2 E}{\delta p^2}(p) =
\alpha_1 \frac{\delta^2 E_1}{\delta p^2}(p) +
\alpha_2 \frac{\delta^2 E_2}{\delta p^2}(p) +
\alpha_3 \frac{\delta^2 E_3}{\delta p^2}(p).
\]
None of three natural gradients listed above is effective for this combined loss functional, since
the inverse of $\frac{\delta^2 E}{\delta p^2}(p)$ looks quite different from
$-\nabla\cdot(p\nabla)$, $\diag(p)$, or $A^{-1}$.

An immediate question is design an efficient natural gradient (or even an approximate one) for the
combined loss $E(p)$. Due to the efficiency considerations, we prefer this natural gradient to have
the following features.
\begin{itemize}
\item It utilizes the Hessian information of $E_1(p)$, $E_2(p)$, and $E_3(p)$ in the design of the
  natural gradient.
\item It avoids forming and/or inverting the Hessian $\frac{\delta^2 E}{\delta p^2}(p)$ in order to
  avoid super-linear costs.
\item The computational cost of computing the natural gradient from $\frac{\delta E}{\delta p}(p)$
  should be of order $O(n\log^c n)$, where $n$ is the number of degrees of freedom used for
  discretizing $p$.
\end{itemize}

The main idea of our approach is to diagonalize $\frac{\delta^2 E}{\delta p^2}(p)$ approximately by
finding a common basis that diagonalizes each of the three terms $\frac{\delta^2 E_1}{\delta
  p^2}(p)$, $\frac{\delta^2 E_2}{\delta p^2}(p)$, and $\frac{\delta^2 E_3}{\delta p^2}(p)$
approximately at the same time. Among various choices, compactly supported wavelets emerge as a
natural candidate because they approximately diagonalize (1) differential operators, (2) diagonal
scaling by functions with sufficient regularity, and also (3) pseudo-differential operators.

\subsection{Related work.}

Fisher-Rao metric is essential to many branches of probability and statistics, as it is invariant
under diffeomorphisms. The study of Fishe-Rao and related metrics has evolved to become the field of
information geometry and we refer to \cite{amari2016information,ay2017information} for detail
discussions. Explicit time-discretization of the Fisher-Rao GD gives rise the mirror descent
algorithms \cite{beck2003mirror,bubeck2015convex,nemirovsky1983problem}, which plays an essential
role in online learning and optimization.

Originated from the theory of optimal transport, Wasserstein metric is defined formally as the
Hessian of the square of the 2nd Wasserstein distance
\cite{villani2003topics,villani2008optimal,santambrogio2015optimal,peyre2019computational}. Starting
from \cite{jordan1998variational,otto2001geometry}, it has been shown that many kinetic-type PDEs
can be viewed as a Wasserstein GD of free energies defined on probability spaces
\cite{carrillo2003kinetic}. In recent years, a parametric version of the Wasserstein metric has been
applied to various applied problems from statistical machine learning
\cite{chen2018natural,li2018natural}.

The quadratic term associated with the Mahalanobis metric appears quite often in partial
differential equation models, for example as the Dirichlet energy or as the interacting free energy
term in the Keller-Segel models \cite{Perthame2006transport}.

A recent paper \cite{ying2020mirror} considers the case where the loss function is the sum of the
Kullback-Leibler divergence and a quadratic interacting term. By adopting a diagonal approximation
of interacting term, it proposes new natural gradient dynamics and develops new mirror descent
algorithms.

\subsection{Contents.}
The rest of this note is organized as follows. Section \ref{sec:algo} proposes a new metric for the
combined loss functional and derives the natural gradient algorithm. In Section \ref{sec:res},
numerical results in 1D and 2D show that the proposed natural gradient outperforms the existing ones
for combined loss functionals. Finally, Section \ref{sec:disc} ends with some discussions on future
work.

\subsection{Data availability statement.}
Data sharing not applicable to this article as no datasets were generated or analyzed during the
current study.

%-----------------------------------
\section{Algorithm} \label{sec:algo}

%-----
\subsection{Metric design} \label{sec:metric}

Consider the 1D problem with $\Omega=[0,1]$ with the periodic boundary condition for simplicity. As
mentioned above for the loss functional $E(p)=\alpha_1 E_1(p)+ \alpha_2 E_2(p)+ \alpha_3 E_3(p)$,
the Hessian can be approximated as follows.
\begin{equation} \label{eq:ddEddp}
  \frac{\delta^2 E}{\delta p^2}(p) =
  \alpha_1\frac{\delta^2 E_1}{\delta p^2}(p) +
  \alpha_2\frac{\delta^2 E_2}{\delta p^2}(p) +
  \alpha_3\frac{\delta^2 E_3}{\delta p^2}(p)
  \approx
  \alpha_1
(-\nabla\cdot(p\nabla))^+ + \alpha_2 \diag\left(\frac{1}{p}\right) + \alpha_3 A.
\end{equation}
For simplicity, assume that the domain $\Omega$ is discretized with a uniform grid with $n$ points
$S = \left\{\frac{0}{n},\frac{1}{n},\ldots,\frac{n-1}{n}\right\}$. A density $p(x)$ for
$x\in\Omega=[0,1]$ can be represented as a vector $p\in\R^n$ with entries denoted by $p_s$ for $s\in
S$.  We denote by $D$ the discrete differential operator. After the discretization, the Hessian
approximation \eqref{eq:ddEddp} takes the following discrete form
\begin{equation} \label{eq:Eppdisc}
  \frac{\delta^2 E}{\delta p^2}(p) \approx \alpha_1 \left(D^\T\diag(p)D\right)^+ + \alpha_2
  \diag\left(\frac{1}{p}\right) + \alpha_3 A.
\end{equation}

As mentioned earlier, the key idea is to diagonalize each of the three terms in \eqref{eq:Eppdisc}
in a compactly supported orthogonal wavelet basis such as the Daubechies wavelets
\cite{daubechies1992ten,mallat1999wavelet}. Let us denote by $W\in \R^{n\times n}$ the matrix such
that its $j$-th column is the $j$-th vector of the wavelet basis. Therefore, $W$ is the matrix for
wavelet reconstruction and its transpose $W^\T$ is the matrix for wavelet decomposition. Notice that
for compactly supported wavelets, $W$ and $W^\T$ are sparse matrices with only $O(n\log n)$ non-zero
entries. Applying $W$ or $W^\T$ to an arbitrary vector of length $n$ takes only $O(n)$ operations by
taking advantages of the filter bank structure of the wavelet basis \cite{mallat1999wavelet}.

Applying the matrices $W^\T$ to the left and $W$ to the right of \eqref{eq:Eppdisc} leads to
\[
W^\T \frac{\delta^2 E}{\delta p^2}(p) W
\approx
\alpha_1 W^\T \left(D^\T\diag(p)D\right)^+ W +
\alpha_2 W^\T \diag\left(\frac{1}{p}\right) W +
\alpha_3 W^\T A W.
\]
The three terms on the right hand side are treated as follows.
\begin{itemize}
\item For the first term, consider first its pseudo-inverse $W^\T D^\T \diag(p) D W$. The
  diagonal entries of $W^\T D^\T\diag(p)DW$ at the $(i,i)$ slot is given by
  \[
  \sum_{s\in S} (DW)_{si} p_s (DW)_{si} = \sum_{s\in S} (DW)_{si}^2 p_s.
  \]
  By defining the matrix $H_1$ with entries given by $(H_1)_{is} = (DW)_{si}^2$, the whole diagonal
  of $W^\T D^\T\diag(p)DW$ can be conveniently written as $H_1 p$, which clearly depends linearly on
  $p$. Taking its pseudo-inverse implies that $W^\T\left(D^\T\diag(p)D\right)^+ W$ can be diagonally
  approximated with $\diag\left(\frac{1}{H_1 p}\right)$.
  
\item For the second term, consider first its pseudo-inverse $W^\T \diag(p) W$. The diagonal
  entry $W^\T \diag(p) W$ at the $(i,i)$ slot is given by
  \[
  \sum_{s\in S} W_{si} p_s W_{si} = \sum_{s\in S} W_{si}^2 p_s.
  \]
  By defining the matrix $H_2 p$ with entries given by $(H_1)_{is} = W_{si}^2$, the whole diagonal
  of $W^\T \diag(p) W$ can be written as $H_2 p$, which is again linear in $p$. Taking its
  pseudo-inverse shows that $W^\T \diag\left(\frac{1}{p}\right) W$ can be diagonally approximated
  with $\diag\left(\frac{1}{H_2 p}\right)$.
\item As opposed to the first two terms, the third term $W^\T A W$ is independent of the density
  $p$. Its diagonal can be precomputed and will be denoted by $h_3 \in \R^n$.
\end{itemize}

%------
Putting the three terms together, we conclude that 
\[
W^\T \frac{\delta^2 E}{\delta p^2}(p) W \approx
\diag\left(
\frac{\alpha_1}{H_1p} + \frac{\alpha_2}{H_2p} + \alpha_3 h_3
\right),
\]
or equivalently
\[
\frac{\delta^2 E}{\delta p^2}(p) \approx
W
\diag\left(
\frac{\alpha_1}{H_1p} + \frac{\alpha_2}{H_2p} + \alpha_3 h_3
\right)
W^\T.
\]
By inverting this approximation, we reach at the metric for the natural gradient
\begin{equation}\label{eqn:metric}
\left( \frac{\delta^2 E}{\delta p^2}(p) \right)^{-1} \approx
W
\diag
\left(
\frac{1}{
  \frac{\alpha_1}{H_1p} + \frac{\alpha_2}{H_2p} + \alpha_3 h_3
}
\right)
W^\T.
\end{equation}

With the metric ready, the ODE for the new natural gradient reads
\begin{equation}\label{eqn:pdot}
  \dot{p} = - W \diag \left( \frac{1}{\frac{\alpha_1}{H_1p} + \frac{\alpha_2}{H_2p} + \alpha_3 h_3}  \right) W^\T
  \frac{\delta E}{\delta p}(p).
\end{equation}
If we denote the wavelet coefficient vector by $c\in \R^n$, i.e., $c=W^\T p$ and $p = Wc$,
\eqref{eqn:pdot} can be written as
\begin{equation}\label{eqn:cdot}
  \dot{c} = -\diag \left( \frac{1}{\frac{\alpha_1}{H_1Wc} + \frac{\alpha_2}{H_2Wc} + \alpha_3 h_3} \right)
  \frac{\delta E}{\delta c}(c).
\end{equation}
In what follows, we simply refer to them as the {\em combined} gradient descent.

The following two claims show that the objects in \eqref{eqn:pdot} and \eqref{eqn:cdot} can be
computed efficiently.
%--
\begin{claim}
  The computational cost of forming and storing the matrices $H_1$ and $H_2$ is $O(n\log n)$.
\end{claim}
\begin{proof}
  Let us recall the definition of the matrices $H_1$ and $H_2$
  \[
  (H_1)_{is} = (DW)_{si}^2,\quad
  (H_2)_{is} = W_{si}^2.
  \]
  Since the wavelets are compactly supported with a constant size support at the finest scale,
  applying the differential operator and taking the element-wise square for a wavelet at scale
  $\ell$ takes $O(n/2^\ell)$ steps. Summing over the wavelets from all scales gives the following
  bound for the total cost:
  \[
  \sum_{\ell=1}^{\log_2 n} 2^\ell \cdot \frac{n}{2^\ell} = O(n\log n).
  \]
\end{proof}

\begin{claim}
  For a density $p\in\R^n$ with $p_i >0$, the computational cost of applying the metric $W \diag
  \frac{1}{ \left( \frac{\alpha_1}{H_1p} + \frac{\alpha_2}{H_2p} + \alpha_3 h_3 \right) } W^\T$
  takes $O(n\log n)$ steps.
\end{claim}
\begin{proof}
  As a consequence from the previous claim, forming $H_1 p$ and $H_2 p$ each takes $O(n\log n)$
  steps. Applying the wavelet decomposition operator $W^\T$ or the reconstruction operator $W$ takes
  $O(n)$ steps by taking advantages of the filter bank construction. Summing them together gives the
  $O(n\log n)$ total cost.
\end{proof}

%-----
\subsection{Time discretization} \label{sec:algodetail}

Let us now describe the time discretization of the natural gradient dynamics \eqref{eqn:pdot}, i.e.,
how to actually use \eqref{eqn:pdot} to find the minimizer. We adopt a backtracking line search
algorithm with Armijo condition \cite{armijo1966minimization}. At time step $k$ with the current
approximation $p^k$, we introduce
\[
s^k = W \diag \left(
\frac{1}{\frac{\alpha_1}{H_1p^k} + \frac{\alpha_2}{H_2p^k} + \alpha_3 h_3}
\right)
W^\T \frac{\delta E}{\delta p}(p^k).
\]
Starting from $\eta=1$, one repetitively halves $\eta$ until
\[
E(p^k-\eta s^k) - E(p_k) \le -\frac{1}{2} \eta s^k\cdot \frac{\delta E}{\delta p}(p^k).
\]
Once it is reached, one sets
\[
p^{k+1}  = p^k - \eta s^k
\]
and move on to the next iteration until convergence.

%% Since $\diag \frac{1}{ \left( \frac{\alpha_1}{H_1p} + \frac{\alpha_2}{H_2p} + \alpha_3 h_3 \right)
%% }$ is diagonal, it is tempting to seek for a mirror descent algorithm based on
%% \eqref{eqn:cdot}. However, it turns out to be quite hard since all the wavelet coefficient $c_i$ are
%% coupling when forming $H_1 Wc$ and $H_2 Wc$. It is not clear how to decouple them.

%-----------------------------------
\section{Numerical results} \label{sec:res}

This section presents several numerical examples to illustrate the efficiency of the combined
gradient descent \eqref{eqn:pdot} for the combined loss functionals.

%----
\subsection{1D}
Consider first the 1D domain $\Omega=[0,1]$ with the periodic boundary condition.  Let $\mu$ be a
reference measure. Among the three terms of the combined loss functional $E(p)=\alpha_1
E_1(p)+\alpha_2 E_2(p)+\alpha_3 E_3(p)$, the first term $E_1(p)$ is a functional close to the square
of the 2nd Wasserstein distance $W_2(p,\mu)$ between $p$ and $\mu$. Because the exact computation of
$W_2^2(p,\mu)$ and its derivative with respect to $p$ is quite non-trivial, we replace $E_1(p)$ with
the square of the weighted semi $H^{-1}$-norm
\[
E_1(p) = \frac{1}{2} \norm{p-\mu}^2_{\dot{H}^{-1}(\mu)},
\]
which is known to be equivalent to the square of the $W_2$ norm \cite{peyre2018comparison}. 
The $\dot{H}^1(\mu)$ for a signed measure $\eps$ is defined as
\[
\norm{\eps}_{\dot{H}^{-1}(\mu)} = \min_{\theta: \nabla\cdot(\mu \theta)=\eps} \int |\theta|^2 \d\mu,
\]
or equivalently
\[
\norm{\eps}_{\dot{H}^{-1}(\mu)} = \sup \left\{ (f,\eps) : \norm{f}_{\dot{H}^1(\mu)} \le 1 \right\},\quad
\norm{f}_{\dot{H}^1(\mu)} = \int |\nabla f|^2 \d\mu.
\]
After discretization, $E_1(p)$ takes the following simple form
\[
E_1(p) = \frac{1}{2} (p-\mu)^\T (D^\T \mu D)^+ (p-\mu).
\]
The second term $E_2(p)$ is the Kullback-Leibler divergence
\[
E_2(p) = \sum_s p_s \log \frac{p_s}{\mu_s}.
\]
Finally, the last term $E_3(p)$ is the Dirichlet energy given by
\[
E_3(p) =\frac{1}{2} (p-\mu)^\T (-\Delta) (p-\mu) = \frac{1}{2} (p-\mu)^\T D^\T D (p-\mu),
\]
so $A = (-\Delta)$. The minimizer of $E(p)$ is equal to $\mu$.

The domain is discretized with $n=512$ grid points. The reference measure $\mu(s) \sim \exp(-V(s))$
with $V(s) = \sin(4\pi s)$ for $s\in S$. The constant factors in front of the three terms are chosen
to be $1$, $10^{-3}$, and $10^{-4}$, respectively, in order to balance the contribution from three
terms so that none of them dominates. We test with four different linear combinations, with results
summarized in Figure \ref{fig:1D}.

\begin{figure}[h!]
  \centering
  \begin{tabular}{cc}
    \includegraphics[width=0.45\textwidth]{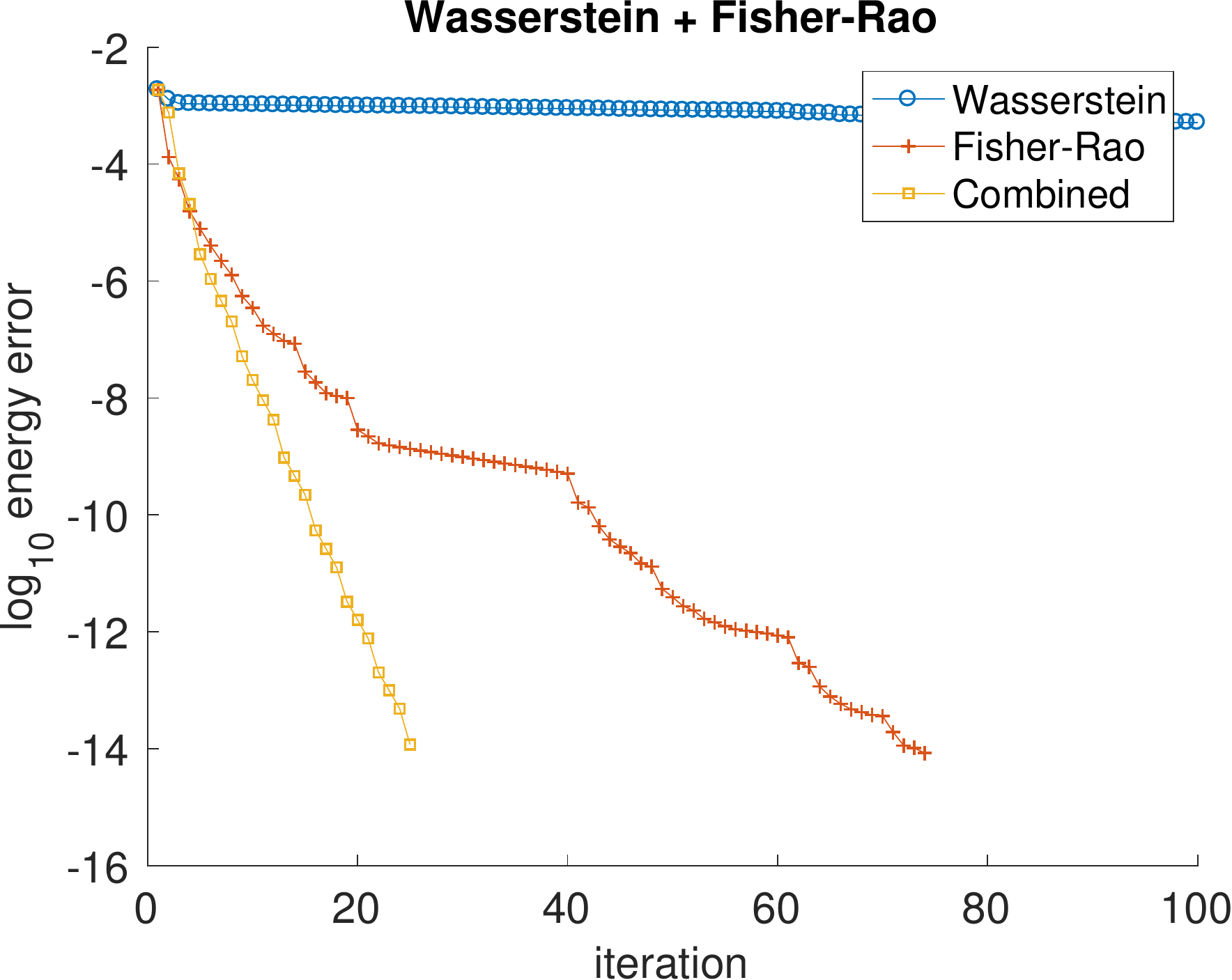} &   \includegraphics[width=0.45\textwidth]{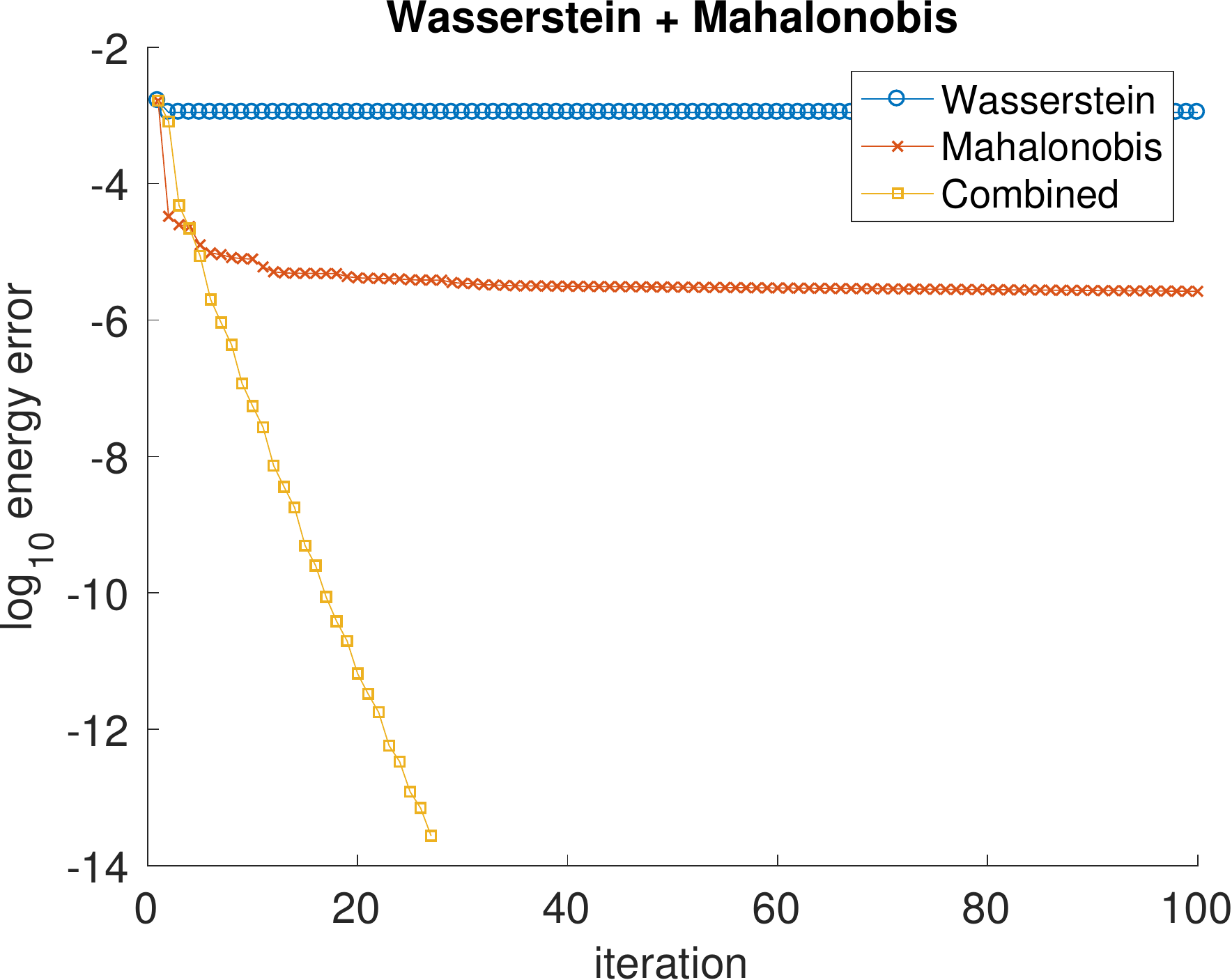}\\
    (1) & (2)\\
    \includegraphics[width=0.45\textwidth]{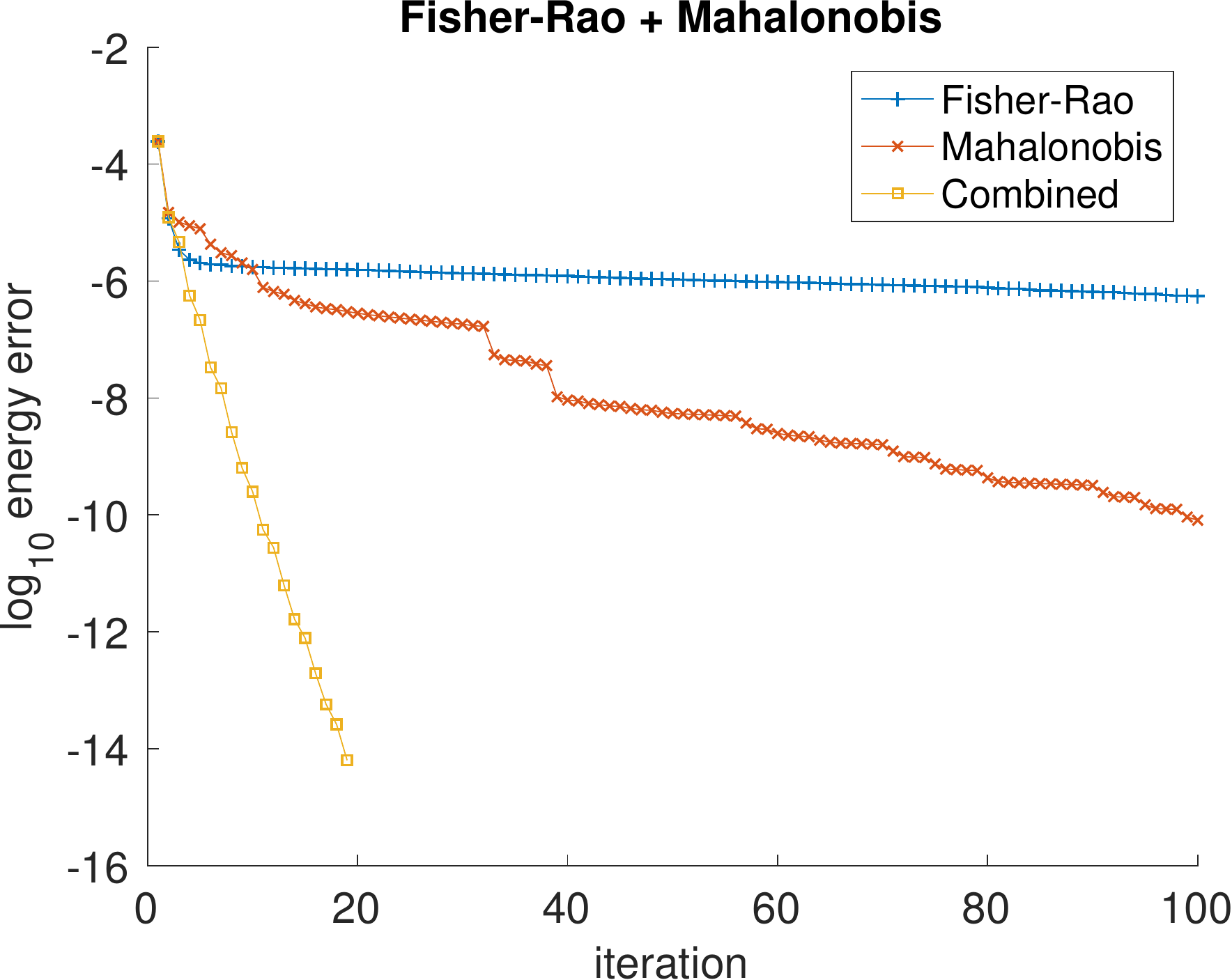} &  \includegraphics[width=0.45\textwidth]{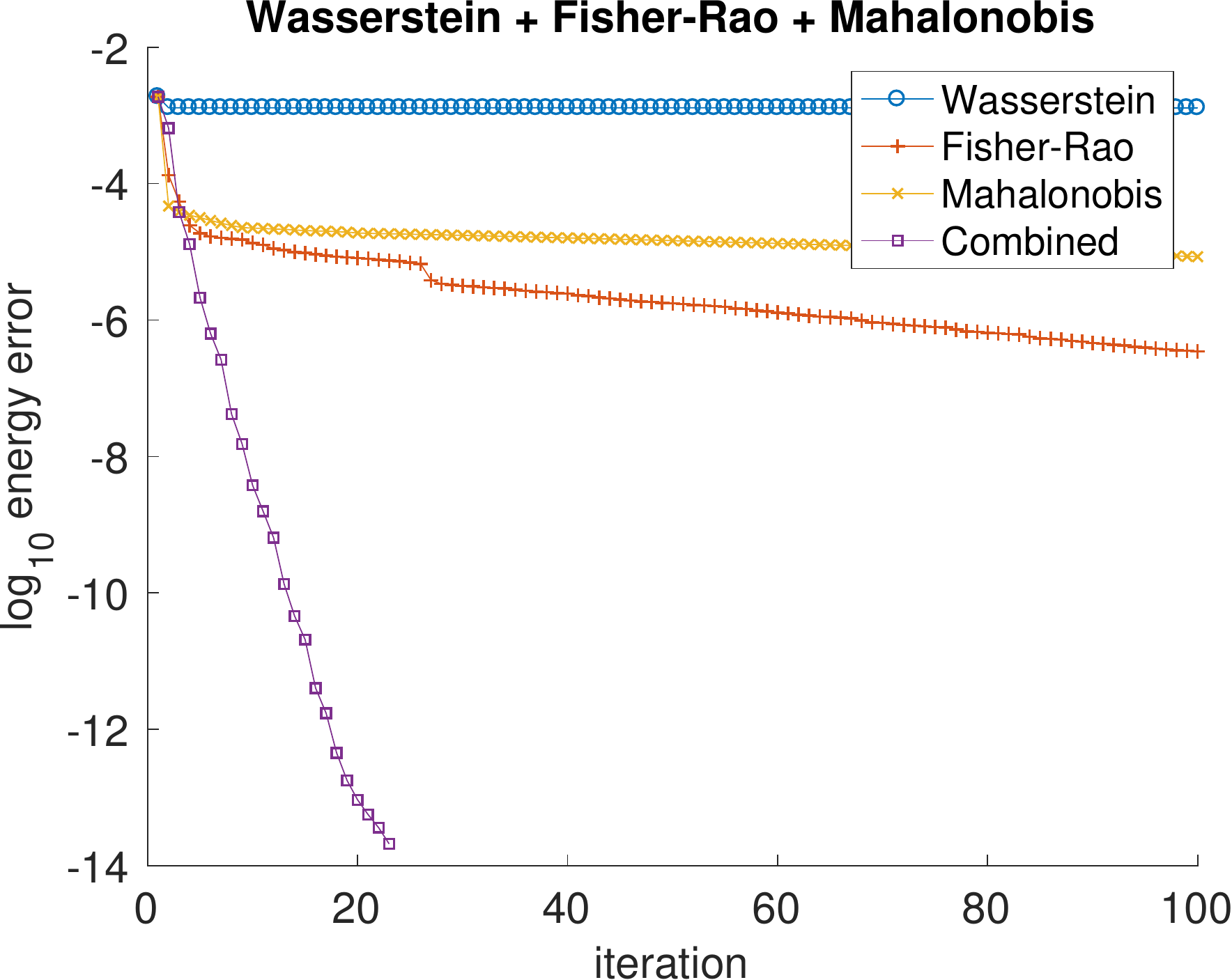}\\
    (3) & (4)
  \end{tabular}
  \caption{ (1) Wasserstein plus Fisher-Rao terms tested with Wasserstein GD, Fisher-Rao GD, and the
    combined natural gradient. (2) Wasserstein plus Mahalanobis terms tested with Wasserstein GD,
    Mahalanobis GD, and the combined natural gradient. (3) Fisher-Rao plus Mahalanobis terms tested
    with Fisher-Rao GD, Mahalanobis GD, and the combined natural gradient. (4) Wasserstein plus
    Fisher-Rao plus Mahalanobis terms tested with Wasserstein GD, Fisher-Rao GD, Mahalanobis GD, and
    the combined natural gradient.}
  \label{fig:1D}
\end{figure}

\begin{enumerate}
\item $(\alpha_1,\alpha_2,\alpha_3) = (1,10^{-3},0)$, i.e., turning off the Mahalanobis
  term. The combined natural gradient converges much more rapidly compared to the Wasserstein GD and
  the Fisher-Rao GD.
\item $(\alpha_1,\alpha_2,\alpha_3) = (1,0,10^{-4})$, i.e., turning off the Fisher-Rao term.
  The combined natural gradient converges much more rapidly compared to the Wasserstein GD and the
  Mahalanobis GD.
\item $(\alpha_1,\alpha_2,\alpha_3) = (0,10^{-3},10^{-4})$, i.e., turning off the Wasserstein
  term. The combined natural gradient converges much more rapidly compared to the
  Fisher-Rao GD and the Mahalanobis GD.
\item $(\alpha_1,\alpha_2,\alpha_3) = (1,10^{-3},10^{-4})$. The combined natural gradient
  converges much more rapidly compared to the Wasserstein GD, the Fisher-Rao GD, and the Mahalanobis
  GD.  
\end{enumerate}

%----
\subsection{2D}
Consider now the 2D domain $\Omega=[0,1]^2$ with periodic boundary condition.  Among the three terms
of the combined loss functional $E(p)=\alpha_1 E_1(p)+\alpha_2 E_2(p)+\alpha_3 E_3(p)$, $E_1(p)$ is
again chosen to be the weighted semi $H^{-1}$-norm
\[
E_1(p) = \frac{1}{2} \norm{p-\mu}_{\dot{H}^{-1}(\mu)}.
\]
After discretization, it takes the following form
\[
E_1(p) = \frac{1}{2} (p-\mu)^\T (D_1^\T \mu D_1 + D_2^\T \mu D_2)^+ (p-\mu)
\]
where $D_1$ and $D_2$ are the derivative operators in the first and the second directions.  $E_2(p)$
is again the Kullback-Leibler divergence
\[
E_2(p) = \sum_s p_s \log \frac{p_s}{\mu_s}.
\]
Finally, $E_3(p)$ is given by
\[
E_3(p) = \frac{1}{2} (p-\mu)^\T (-\Delta) (p-\mu) = \frac{1}{2} (p-\mu)^\T (D_1^\T D_1 + D_2^\T D_2) (p-\mu) 
\]
so $A = (-\Delta)$.

The domain is discretized with $n=64$ grid point in each direction. $\mu(s_1,s_2)\sim
\exp(-V(s_1,s_2))$ with $V(s_1,s_2) = \sin(4\pi s_1)\sin(4\pi s_2)$. The constant factors of the
three terms are set to be $1$, $3\cdot10^{-4}$, and $10^{-4}$ in order to balance the contribution
from them so that no one dominates. We test with four different linear combinations, with the
results summarized in Figure \ref{fig:2D}.

\begin{figure}[h!]
  \centering
  \begin{tabular}{cc}
    \includegraphics[width=0.45\textwidth]{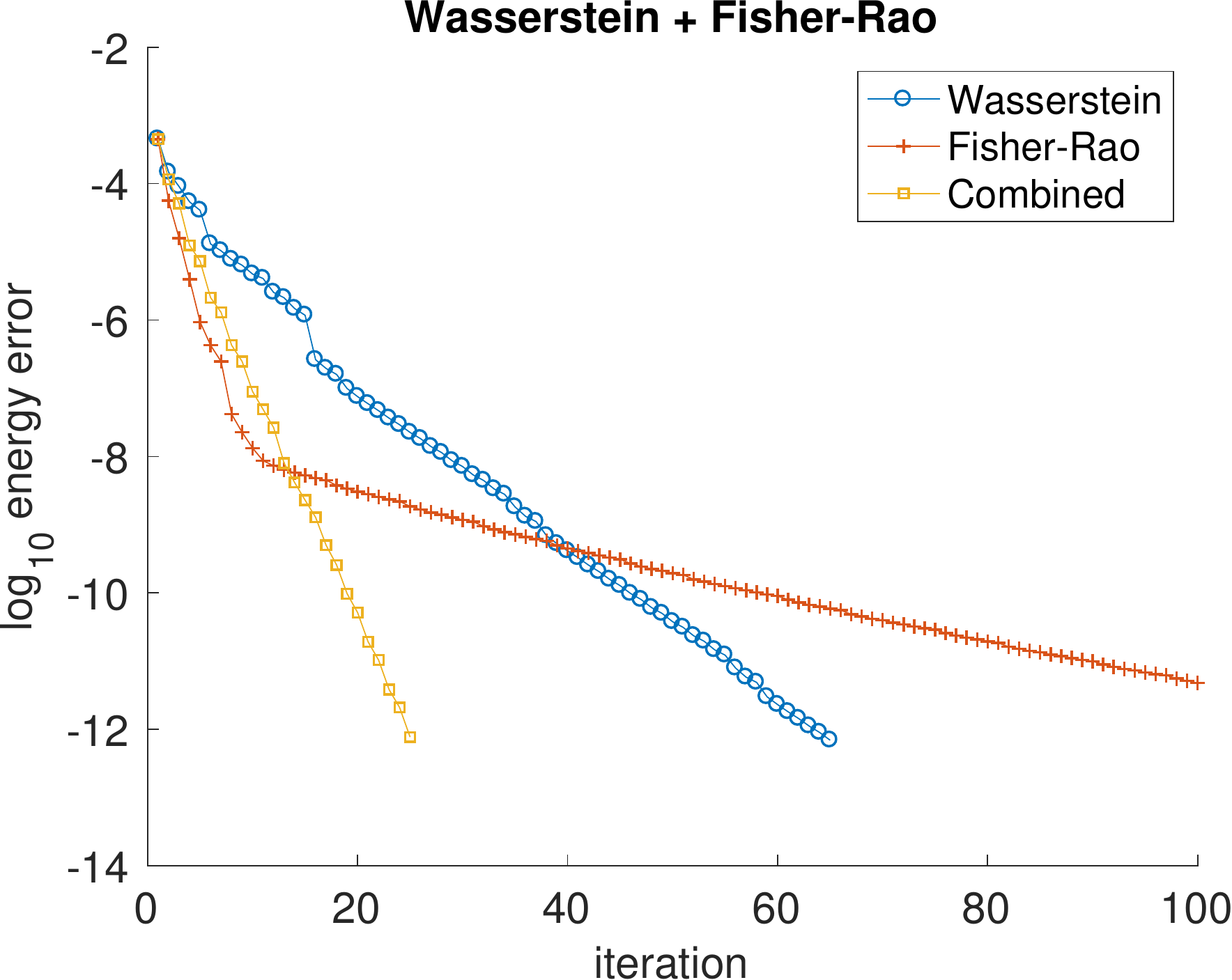} &    \includegraphics[width=0.45\textwidth]{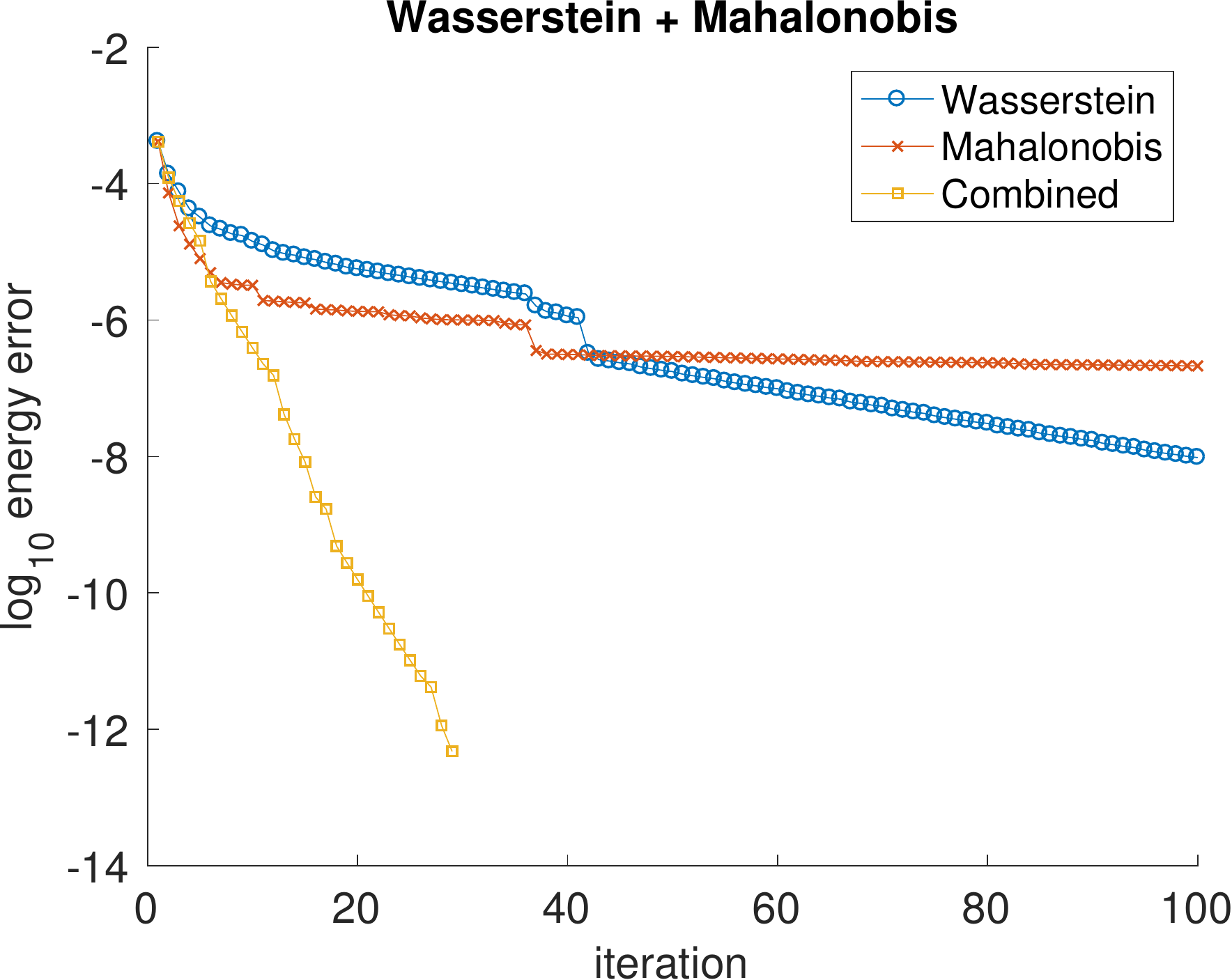}\\
    (1) & (2)\\
    \includegraphics[width=0.45\textwidth]{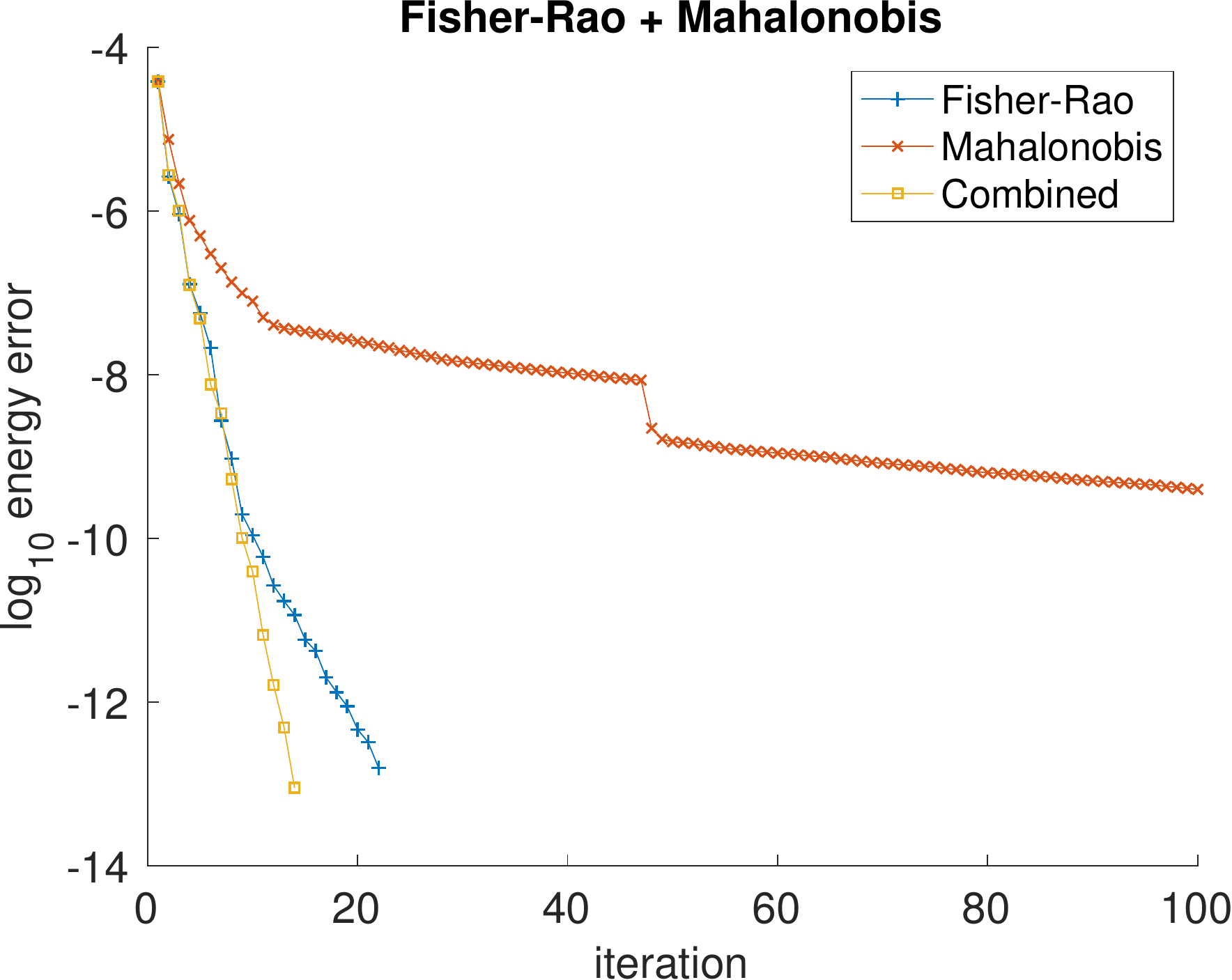} &    \includegraphics[width=0.45\textwidth]{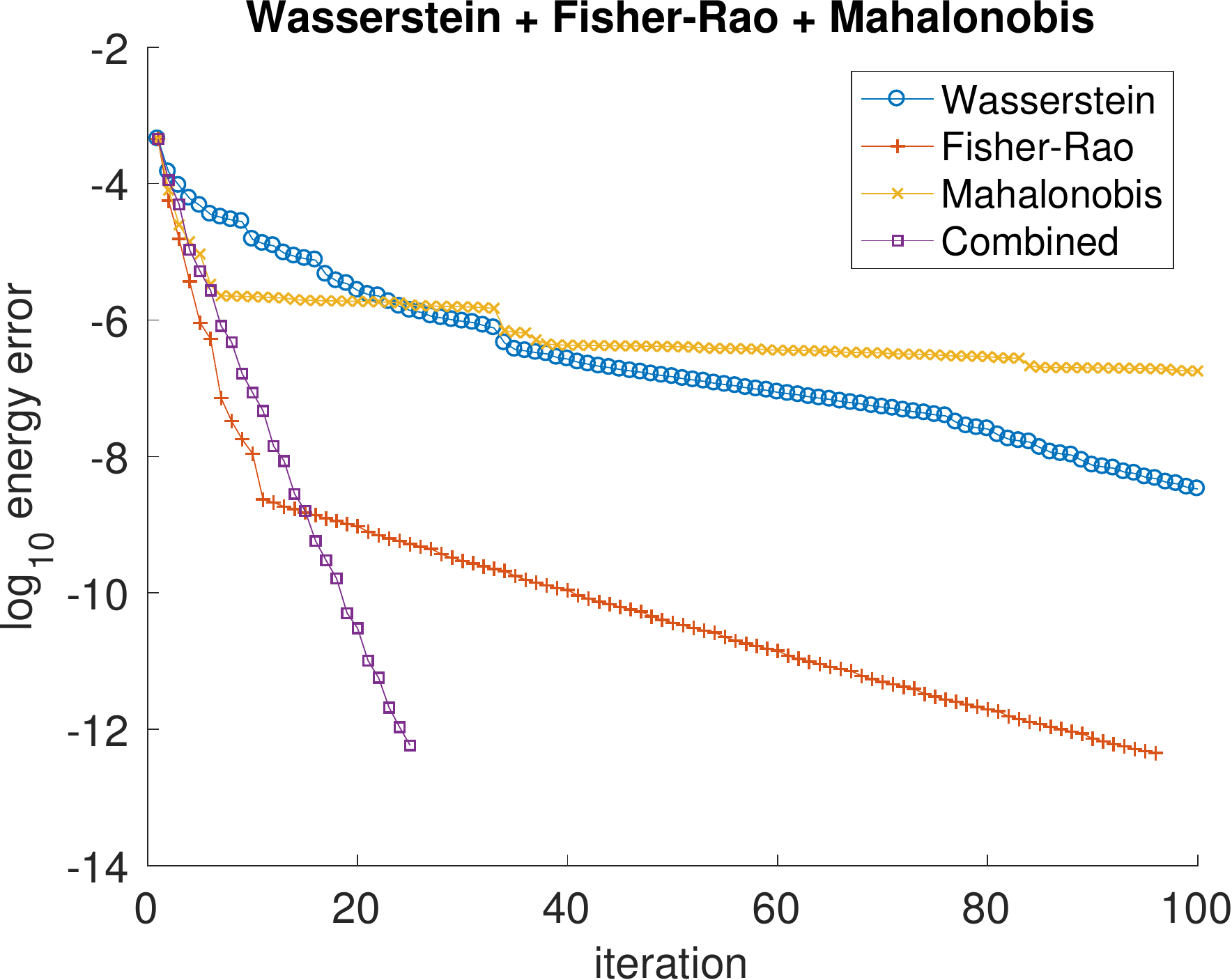}\\
    (3) & (4)
  \end{tabular}
  \caption{ (1) Wasserstein plus Fisher-Rao terms tested with Wasserstein GD, Fisher-Rao GD, and the
    combined natural gradient. (2) Wasserstein plus Mahalanobis terms tested with Wasserstein GD,
    Mahalanobis GD, and the combined natural gradient. (3) Fisher-Rao plus Mahalanobis terms tested
    with Fisher-Rao GD, Mahalanobis GD, and the combined natural gradient. (4) Wasserstein plus
    Fisher-Rao plus Mahalanobis terms tested with Wasserstein GD, Fisher-Rao GD, Mahalanobis GD, and
    the combined natural gradient.}
  \label{fig:2D}
\end{figure}

\begin{enumerate}
\item $(\alpha_1,\alpha_2,\alpha_3) = (1,3\cdot10^{-4},0)$, i.e., turning off the Mahalanobis
  term. The combined natural gradient converges much more rapidly compared to the Wasserstein GD and
  the Fisher-Rao GD.
\item $(\alpha_1,\alpha_2,\alpha_3) = (1,0,10^{-4})$, i.e., turning off the Fisher-Rao term. The
  combined natural gradient converges much more rapidly compared to the Wasserstein GD and the
  Mahalanobis GD.
\item $(\alpha_1,\alpha_2,\alpha_3) = (0,3\cdot10^{-4},10^{-4})$, i.e., turning off the
  Wasserstein term. The combined natural gradient converges much more rapidly compared to the
  Fisher-Rao GD and the Mahalanobis GD.
\item $(\alpha_1,\alpha_2,\alpha_3) = (1,3\cdot10^{-4},10^{-4})$. The combined natural gradient
  converges much more rapidly compared to the Wasserstein GD, the Fisher-Rao GD, and the Mahalanobis
  GD.
\end{enumerate}

%-----------------------------------
\section{Discussions} \label{sec:disc}

This note proposes a new natural gradient for minimizing combined loss functionals by using
diagonal approximation in the wavelet basis. There are a few open questions. First, so far we have
considered regular domains in one and two dimensions with periodic boundary condition. One direction
is to extend this to more general domains using more sophisticated wavelet bases.

Second, we have assumed that the probability density $p$ is non-vanishing everywhere in deriving the
interpolating natural gradient metric. It is an important question whether one can remove this
condition in order to work with more general probability densities.

Third, the dynamics in the wavelet coefficients \eqref{eqn:cdot} enjoys a diagonal metric. It is
tempting to ask whether it is possible to design a mirror descent algorithm. Due to the coupling
between different wavelet coefficients in the metric computation $H_1 Wc$ and $H_2 Wc$, this seems
quite difficult. An interesting observation is that the metric of the coarse scale wavelet
coefficients is nearly independent of the values of the fine scale coefficients, while the metric of
the fine scale ones depends heavily on the values of the coarse scale ones. This naturally brings
the question of whether the combined metric (or even the Wasserstein metric) has an approximate
multiscale structure. The wavelet analysis has played an important role understanding the earth
mover distance metric $W_1$ \cite{indyk2003fast,shirdhonkar2008approximate}. It seems that it might
also play a role in understanding the $W_2$ metric.

\bibliographystyle{abbrv}

\bibliography{ref}

\end{document}